\documentclass[11pt, onesided, reqno]{amsart}

\usepackage{amssymb}
\usepackage{amsmath}
\usepackage{color}
\usepackage{enumerate}
\usepackage{graphicx}

\evensidemargin\oddsidemargin

\newtheorem{theorem}{Theorem}
\newtheorem{proposition}[theorem]{Proposition}
\newtheorem{corollary}[theorem]{Corollary}
\newtheorem{lemma}[theorem]{Lemma}

\theoremstyle{definition}

\newtheorem{remark}[theorem]{Remark}

\newcommand{\eqnsection}{
\renewcommand{\theequation}{\thesection.\arabic{equation}}
    \makeatletter
    \csname  @addtoreset\endcsname{equation}{section}
    \makeatother}
\eqnsection

\def\e{\mathbf{e}}
\def\g{\mathbf{g}}
\def\E{\mathbb{E}}

\def\N{\mathbb{N}}

\def\R{\mathbb{R}}

\def\Pb{\mathbb{P}}

\def\F{\mathcal{F}}

\newcommand{\equi}{\; \mathop{\sim}\limits}
\def\={{\,\;\mathop{=}\limits^{\text{(law)}}\;\,}}

\def\qed{\hfill$\square$}

\allowdisplaybreaks

\makeatletter
\@namedef{subjclassname@2020}{%
  \textup{2020} Mathematics Subject Classification}
\makeatother

\begin{document}

\title[]{Survival and maximum of spectrally negative branching L\'evy processes with absorption}
\author[Christophe Profeta]{Christophe Profeta}

\address{
Universit\'e Paris-Saclay, CNRS, Univ Evry, Laboratoire de Math\'ematiques et Mod\'elisation d'Evry, 91037, Evry-Courcouronnes, France.
 {\em Email} : {\tt christophe.profeta@univ-evry.fr}
  }


\subjclass[2020]{}

\begin{abstract} 
We consider a spectrally negative branching L\'evy process where the particles undergo dyadic branching and are killed when entering the negative half-plane. 
The purpose of this short note is to give conditions under which this process dies out a.s., and then study the asymptotics of its all-time maximum. 
\end{abstract}

\maketitle

\section{Introduction}

Consider a spectrally negative branching L\'evy process $X$ where the particles undergo dyadic branching at rate $\beta>0$, and are killed when entering the negative  half-line. More precisely, starting at $t=0$ from an initial particle located at $a>0$, the reproduction and dispersion mechanisms are defined as follows: 
\begin{enumerate}
\item { The spatial movements and the branching mechanism are independent.}
\item \emph{Branching}: each particle waits for an exponential random time $\e$ of parameter $\beta$ and, if still alive, splits into two independent particles whose lives start at the location of their ancestor.
\item \emph{Spatial motion}: between branching events, the particles move independently according to a spectrally negative L\'evy process $L$  but are killed when going below~0.
\end{enumerate}
Let us denote by $\textbf{M}$ the all-time maximum location ever reached by a particle during the whole life of the process. The purpose of this paper is to give conditions under which the process dies out a.s. and then study the distribution of the (then finite) random variable~$\textbf{M}$.\\

As explained in \cite{MaSc},  such a model may describe for instance the evolution of a population in which the positions of the particles represent their fitness. The fitness of each individual evolves as a spectrally negative L\'evy process, and when branching, the initial fitness of a child is the same as that of his parent. The presence of negative jumps makes it possible to model diseases or accidents that generate an abrupt drop in the fitness, and individuals die when their fitness goes below 0.
\\

In the following, we shall assume that all the processes and random variables are defined on the same probability space $(\Omega, \F, \Pb)$, and we shall denote  by $\Pb_a$, with an abuse of notation, both the laws of $X$ and $L$ when started from $a$.
For {$\lambda \in \R$, $\lambda \geq0$}, the Laplace exponent of $L$ is given by 
\[\Psi(\lambda) =\ln \E_0\left[e^{\lambda L_1}\right]=d\lambda +  \frac{\eta^2}{2}\lambda^2 + \int_{-\infty}^0 \left(e^{\lambda x}- 1 - \lambda x 1_{\{|x|<1\}}\right) \nu(dx)\]
where $d \in \R$ is the drift coefficient, $\eta \in \R$ the Gaussian coefficient and the L\'evy measure $\nu$ satisfies $ \int_{-\infty}^0 (x^2\wedge 1)\, \nu(dx)<+\infty$. 
We assume that the one-dimensional distributions of $L$ are absolutely continuous, i.e., 
\[\Pb_0(L_t\in dx) \ll dx, \qquad \text{for every }t>0\]
and we exclude the case where $-L$ is a subordinator (in which case $\textbf{M}$ is finite a.s. and equal to $a$).  As a consequence, the function $\Psi$ is strictly convex and tends to $+\infty$ as $\lambda\rightarrow +\infty$. This implies that for any $q\geq0$, the equation $\Psi(\lambda)=q$ admits at most two solutions, and we denote by $\Phi(q)$ the largest one:
\[\Phi(q) = \sup\{\lambda\geq 0,\, \Psi(\lambda)=q\}.\]
More precisely:

\begin{enumerate}
\item When $\Psi^\prime(0^+)\geq0$, we have $\Phi(0)=0$ and the function $\Phi$ is well-defined on $[0,+\infty)$.
\item When $\Psi^\prime(0^+)<0$, we have $\Phi(0)>0$. In this case, $\Psi$ being strictly convex, it admits a unique minimum at $\lambda_\ast$ which is such that $\Psi^\prime(\lambda_\ast)=0$ and $\Psi^{\prime\prime}(\lambda_\ast) = \eta^2 + \int_{-\infty}^{0} x^2 e^{\lambda_\ast x} \nu(dx)>0$. As a consequence, the function  $\Phi$ is well-defined on $[-q_\ast,+\infty)$, and $\Phi(-q_\ast) = \lambda_\ast$.
\end{enumerate}

\noindent
Our first result concerns the survival of this branching process. 
\begin{theorem}\label{theo:1}
The branching process dies out a.s. if and only if $\Psi^\prime(0^+)<0$ and $\beta \leq q_\ast$. In this case, the asymptotics of the extinction time $\zeta$ satisfy:
\begin{equation}\label{eq:asympzeta}
\lim_{t\rightarrow +\infty}\frac{1}{t}\ln \Pb_a(\zeta>t)  = \beta-q_\ast.
\end{equation}
\end{theorem}

{Note that for a standard  L\'evy process $L$,  the condition $\Psi^\prime(0^+)<0$ implies that $L$ converges towards $-\infty$. In other words, each particle of the branching process is attracted in some sense to $-\infty$.} The second condition $\beta\leq q_\ast$ essentially states that the decay of the particles towards $-\infty$ should be fast enough to compensate the reproduction rate $\beta$.\\

When $\Psi^\prime(0^+)<0$ and $\beta \leq q_\ast$,  the above theorem implies that the random variable $\textbf{M}$ is finite, and we are now interested in studying {its tail decay}. To this end, let us introduce the following scale functions $W^{(q)}$ which are defined for  $q\geq 0$ by:
\begin{equation}\label{eq:W}
\int_0^{+\infty}e^{-\lambda x} W^{(q)}(x)dx = \frac{1}{\Psi(\lambda)-q},\qquad \lambda > \Phi(q).
\end{equation}
It is known that $W^{(q)}$ is an increasing function that is null on $(-\infty,0)$. Furthermore, for every $a>0$, the function $q\rightarrow W^{(q)}(a)$ may be extended analytically on $\mathbb{C}$, see Kyprianou \cite[Section 8.3]{Kyp}.\\

\begin{theorem}\label{theo:2}
Assume that $\Psi^\prime(0^+)<0$ and $\beta\leq q_\ast$ so that $\textbf{M}$ is a.s. finite.
\begin{enumerate}
\item If $\beta<q_\ast$, there exists a constant $\kappa_\beta$ independent of the starting point $a$ such that:
\[\Pb_a\left(\textbf{M}\geq x\right) \equi_{x\rightarrow +\infty} \kappa_\beta W^{(-\beta)}(a)  e^{-\Phi(-\beta) x}.\]
\item If $\beta=q_\ast$,  there exists a constant $\kappa_{q_\ast}$ independent of  the starting point $a$ such that:
\[\Pb_a\left(\textbf{M}\geq x\right) \equi_{x\rightarrow +\infty} \kappa_{q_\ast}  \frac{W^{(-q_\ast)}(a) }{x}e^{-\Phi(-q_\ast) x}.\]
\end{enumerate}
\end{theorem}

It might be surprising to observe that there is an extra decreasing factor in the critical case, which implies that the maximum will be smaller than one might expect by simply letting $\beta\uparrow q_\ast$ in the subcritical case. This phenomenon is in accordance with what is known about the extinction time in the Brownian case.  Indeed, if $L$ is a Brownian motion with  drift, the asymptotics of the extinction time $\zeta$ are explicitly given by 
\[\begin{cases}
\displaystyle \Pb_a(\zeta > t) = C_\beta(a) t^{-3/2}e^{(\beta-q_\ast) t}  & \quad \text{if } \beta <q_\ast,\\
\displaystyle \Pb_a(\zeta > t) =  C_{q_\ast}(a)e^{-(3\pi^2 q_\ast t)^{1/3}   } & \quad \text{if } \beta =q_\ast,
\end{cases}\]
where $C_\beta(a)>0$. These asymptotics were obtained through the successive works of Kesten \cite{Kes}, Harris \& Harris \cite{HaHa},  Berestycki, Berestycki \& Schweinsberg \cite{BBS} and Maillard \& Schweinsberg \cite{MaSc}. In particular, we observe the same phenomenon: in the critical case, the branching Brownian motion will die out faster than one might expect by simply passing to the limit in the subcritical case.\\

Let us finally consider the same spectrally negative branching L\'evy process, but without killing. We denote by $\mathcal{M}$ its all-time maximum. Using the invariance by translation of L\'evy processes, we deduce by letting the killing barrier go down to $-\infty$ the following corollary.
\begin{corollary}\label{cor}
The maximum $\mathcal{M}$ of a (free) spectrally negative branching L\'evy process is a.s. finite if $\Psi^\prime(0^+)<0$ and  $\beta\leq q_\ast$. In this case, there exists a constant $\kappa_\beta>0$ such that 
\[\Pb_0\left(\mathcal{M}\geq x\right) \equi_{x\rightarrow +\infty} \kappa_\beta  e^{-\Phi(-\beta) x}.\]
\end{corollary}
Note that this result was already known in the Brownian case, see \cite[Prop. 4]{BBHSM} where the proof essentially relies on the well-known associated KPP equation.\\

\noindent
We finally mention that we have chosen to state our results in the case of dyadic branching, but our method also applies to more general offspring distributions, assuming at least the existence of a finite third moment. In this case, an extra step would be required, as the integrands in the Feynman-Kac representations (see Propositions \ref{prop:zeta} and \ref{prop:M}) would no longer be linear, and one would have to control the remainders as was done in \cite{Pro}.

\section{Finiteness of the extinction time}

\subsection{Proof of the ``if"  part of Theorem \ref{theo:1}}

Assume that $\Psi^\prime(0^+)<0$ and $\beta\leq q_\ast$. Let us denote by $N_t$ the number of particles alive at time $t$, and $\zeta=\inf\{t\geq0,\; N_t=0\}$ the extinction time. Since $N_t$ is integer-valued, we have, using the  many-to-one formula,
\begin{equation}\label{eq:boundv}
\Pb_a(\zeta> t) =\Pb_a(N_t>0) \leq \E_a[N_t] = \Pb_a(\tau_0^->t) e^{\beta t},
\end{equation}
where, for $x\in \R$, 
\[\tau_x^- = \inf\{t\geq 0,\, L_t\leq x\}. \]
But, from  Kyprianou \& Palmowski \cite[Corollary 4]{KyPa}, there exists a function $\kappa$ such that\footnote{Note that, with the notations of \cite{KyPa}, a spectrally negative L\'evy process such that $\Psi^\prime(0^+)<0$  is of class A with $\gamma=2$.}
\begin{equation}\label{eq:asympT0}
\Pb_a(\tau_0^- >t) \equi_{t\rightarrow +\infty}  \kappa(a)  t^{-3/2} e^{-q_\ast t}.
\end{equation}
Passing to the limit in (\ref{eq:boundv}), we thus deduce that $\lim\limits_{t\rightarrow +\infty}\Pb_a(\zeta> t) =0$ since $\beta\leq q_\ast$. As a consequence, $\Pb_a(\zeta=+\infty) = 0$, i.e., the process dies out a.s.

\subsection{A representation formula for  the extinction time}

To prove the converse part, we shall rely on the following representation formula for the distribution of $\zeta$. 
 \begin{proposition}\label{prop:zeta}
Define $v(a,t) = \Pb_a(\zeta>t)$. Then $v$ admits the representation:
\[
v(a,t)= e^{\beta t} \E_a\left[1_{\{\tau_0^->t\}} e^{ - \beta \int_0^t v(L_{s}, t-s) ds}   \right] .
\]
\end{proposition}
In the Brownian case, this result was already used in \cite{HaHa} to study the asymptotic tail of $\zeta$. One classical way to prove such representation is to write down a pseudo-differential equation satisfied by $v$ (involving the generator of the L\'evy process $L$), and then apply a Feynman-Kac theorem or a martingale argument. We shall propose here another approach based on solving a Fredholm integral equation of the second kind.

\begin{proof}
We start by applying the Markov property at the first branching event $\e$:
\begin{align*}
1-v(a,t) &= \Pb_a(\zeta \leq t, \,t<\e)+ \Pb_a(\zeta \leq t,\, \tau_0^- \leq \e, t \geq \e) + \Pb_a(\zeta \leq t,\, \tau_0^- > \e,\, t\geq \e)\\
&= \Pb_a(\tau_0^-\leq t,\, t<\e) + \Pb_a(\tau_0^- \leq \e, \,t\geq \e)  + \E_a\left[1_{\{\tau_0^->\e,\, t>\e\}}(1-v(L_\e, t-\e))^2 \right].
\end{align*}
Here, $\tau_0^-$ denotes the first passage time below 0 of the initial particle $L$, and is independent of $\e$. Developing the square yields the integral equation:
\[v(a,t) = e^{-\beta t} \Pb_a(\tau_0^->t) + 2 \E_a\left[1_{\{\tau_0^->\e,\, t>\e\}} v(L_\e, t-\e)\right] - \E_a\left[1_{\{\tau_0^->\e,\, t>\e \}} v^2(L_\e, t-\e)\right].\]
To obtain a representation of $v$ through this equation, we shall write down a Riemann-Liouville series.
Define the linear operator $T$ acting on positive and bounded functions $f:  \R\times[0,+\infty) \rightarrow [0,+\infty)$ by
\[T[f](x,t) =  \E_x\left[1_{\{\tau_0^->\e,\, t>\e\}}  \left(2- v(L_\e, t-\e)\right) f(L_\e, t-\e) \right] \]
so that $v$ is a solution of the Fredhlom integral equation:
\[v(x,t) =  e^{-\beta t}\Pb_x(\tau_0^->t)  + T[v](x,t).\]
Let us set to simplify the notation  $\varphi(x,t) =  e^{-\beta t} \Pb_x(\tau_0^->t)$. By iteration, we deduce that for $n\geq2$,
\[v(x,t)=  \varphi(x,t) + \sum_{k=1}^{n-1} T^{\circ( k)}[\varphi](x,t) + T^{\circ (n)}[v](x,t) \]
where $T^{\circ(k)}$ denotes the $k^{\text{th}}$ composition of $T$ with itself: 
\[T^{\circ (1)} = T\qquad \text{ and for $k\geq 2$,}\qquad T^{\circ(k)}[f] = T^{\circ(k-1)}[T[f]].\]
We now show that for any pair $(x,t)\in (0,+\infty)^2$, 
\[\lim_{n\rightarrow +\infty} T^{\circ (n)}[v](x,t) =0.\]
Using the upper bound (\ref{eq:boundv}) for $v$, we have 
\[
T[v](x,t)\leq2 \E_x\left[e^{\beta (t-\e)} \Pb_{L_\e}(\tau_0^->t- \e)  1_{\{\tau_0^->\e\}} 1_{\{t> \e\}}\right].\]
Applying the Markov property at the time $\e$, this is further equal to 
\begin{align*}
T[v](x,t)\leq 2 \E_x\left[e^{\beta (t-\e)} 1_{\{\tau_0^->t\}} 1_{\{t> \e\}}\right]&= 2e^{\beta t}\Pb_x(\tau_0^->t) \E[e^{-\beta \e}1_{\{t>\e\}}] \\
 &=e^{\beta t}\Pb_x(\tau_0^->t) (1-e^{-2\beta t})   \end{align*}
since $\e$ follows an exponential distribution of parameter $\beta$. By iteration, we deduce that 
\[T^{\circ( n)}[v](x,t)  \leq (1-e^{-2\beta t})^n e^{\beta t}\Pb_x(\tau_0^->t) \xrightarrow[n\rightarrow +\infty]{} 0.\]
As a consequence, we obtain the representation for any $(x,t) \in (0,+\infty)^2$,
\[v(x,t)=  \varphi(x,t) + \sum_{k=1}^{+\infty} T^{\circ (k)}[\varphi](x,t).\]
We then compute the convolutions appearing in the sum.
\begin{lemma}\label{lem:rec}
Let $(\e_k, \, k\in \N)$ be a sequence of i.i.d. exponential r.v. with parameter $\beta$. For any $n\in \N$, we set  $\g_n = \sum_{k=1}^n \e_k$. Then:
\[
T^{\circ (n)}[\varphi](x,t)
=e^{-\beta t }\E_x\left[1_{\{\tau_0^->t\}}1_{\{\g_n\leq t\}}  e^{\beta\g_n} \prod_{i=1}^{n}  (2-v(L_{\g_i}, t-\g_i))  \right].
\]
\end{lemma}

\begin{proof}
For $n=1$, we have applying the Markov property 
\begin{align*}
T[\varphi](x,t) &=  \E_x\left[1_{\{\tau_0^->\e,\, t>\e\}}  \left(2- v(L_\e, t-\e)\right) \Pb_{L_\e}(\tau_0^->t-\e) e^{-\beta (t-\e)}     \right]\\
&= e^{-\beta t} \E_x\left[1_{\{\tau_0^->\e,\, t>\e\}}  \left(2- v(L_\e, t-\e)\right) 1_{\{\tau_0^->t\}} e^{\beta \e}     \right]
\end{align*}
which is the expected result.  Assume now that the formula holds for some $n\in \N$. Applying the Markov property, we have $\Pb_a$ a.s.:
\[
T^{\circ (n)}[\varphi](L_\e,t-\e) =e^{-\beta t } \E_a\left[1_{\{\tau_0^->t\}}1_{\{\e+\g_n\leq t\}}  e^{\beta (\e+\g_n)}\prod_{i=1}^{n}  (2-v(L_{\e+\g_i}, t-\e-\g_i)) \bigg| \F_\e \right]
\]
where $(\F_t, t\geq0)$ denotes the natural filtration of $L$, and the first branching event $\e$ is independent of $L$ and of the sequence $(\g_i, \, 1\leq i\leq n)$, {i.e., of $(\e_k, 1\leq k\leq  n)$.}
Then, using the tower property of conditional expectation, we obtain:
\begin{align*}
T^{\circ (n+1)}[\varphi](x,t) &= \E_x\left[1_{\{\tau_0^->\e,\, t>\e\}}  \left(2- v(L_\e, t-\e)\right) T^{\circ(n)}[\varphi](L_\e, t-\e) \right] \\
&=e^{-\beta t }\E_x\left[1_{\{\tau_0^->t\}}1_{\{\g_{n+1}\leq t\}}  e^{\beta\g_{n+1}} \prod_{i=1}^{n+1}  (2-v(L_{\g_i}, t-\g_i))  \right]
\end{align*}
which proves Lemma \ref{lem:rec} by induction.
\end{proof}
\noindent
Recall now that  for $n\geq1$, the r.v. $\g_n$ is Gamma distributed with parameters $n$ and $\beta$:
\[\Pb\left(\g_n\in dr\right) = \frac{\beta^n}{(n-1)!} r^{n-1} e^{-\beta r}dr, \qquad r>0,\]
and that, conditionally  on $\{\g_n= r\}$, the partial sums are distributed as ordered uniform random variables, i.e., 
\[\left(\g_i, 1\leq i\leq n-1\right)\; \mathop{=}^{\text{(law)}} \;r\times \left(U_i, 1\leq i\leq n-1\right).\]
where the $(U_i)$ are uniform r.v. on $[0,1]$, { independent of everything else,} and  such that $U_1\leq U_2\leq \ldots\leq U_{n-1}$. As a consequence, conditioning on $\{\g_n= r\}$, we deduce that the expectation in Lemma \ref{lem:rec} equals
\[ e^{-\beta t}  \int_0^t  \E_x\left[1_{\{\tau_0^->t\}} (2-v(L_{r}, t-r))   \prod_{i=1}^{n-1}  (2-v(L_{U_i r }, t-U_i r))    \right]  \frac{\beta^n r^{n-1}}{(n-1)!}dr.   \]
Computing the expectations with respect to the $(U_i)$, we obtain:
\[ \frac{\beta^n}{(n-1)!}e^{-\beta t}  \int_0^t  \E_x\left[1_{\{\tau_0^->t\}} (2-v(L_{r}, t-r))  \left(\int_0^r  (2-v(L_{s}, t-s)) ds\right)^{n-1}   \right]  dr.   \]
Finally, plugging this expression into the series representation of $v$ and computing the remaining integral yields the result of Proposition \ref{prop:zeta}:
\begin{align*}
v(a,t) &=  \Pb_a(\tau_0^->t) e^{-\beta t}  +  e^{-\beta t}  \beta \int_0^t  \E_a\left[1_{\{\tau_0^->t\}} (2-v(L_{r}, t-r)) e^{\beta \int_0^r  (2-v(L_{s}, t-s)) ds}   \right] dr  \\
&= \E_a\left[1_{\{\tau_0^->t\}} e^{\beta \int_0^t  (1-v(L_{s}, t-s)) ds}   \right].
\end{align*}
\end{proof}

\subsection{Proof of the ``only if" part of Theorem \ref{theo:1} when $\Psi^\prime(0^+)<0$}

{ We now assume that  $\Psi^\prime(0^+)<0$ and that the branching process dies out a.s., i.e., that $\lim\limits_{t\rightarrow +\infty} v(a,t)=\Pb_a(\zeta=+\infty)=0.$ Then, a standard application of the Markov property shows that we necessarily have $\lim\limits_{t\rightarrow +\infty}v(x,t) = 0$ for any $x>0$, and 
we shall prove that this implies that $\beta\leq q_\ast$. Let us set for $x\in \R$,
\[ \tau_x^+ = \inf\{t\geq 0,\, L_t\geq x\}.\]
Since $x\rightarrow v(x,s)$ is increasing and $L$ has no positive jumps, we have for $x>a$:
\[
v(a,t)
 \geq e^{\beta t} \E_a\left[1_{\{\tau_0^->t\}} e^{ - \beta \int_0^t v(L_{t-s}, s) ds} 1_{\{\tau_x^+>t\}}  \right]\geq e^{\beta t - \beta \int_0^t v(x, s) ds} \Pb_a\left(\tau_0^-\wedge \tau_x^+>t\right).
\]
From  Bertoin \cite[Theorem 2]{BerScale}, since we have assumed that the transition density of $L$ is absolutely continuous, the asymptotics of the exit time $\tau_0^-\wedge \tau_x^+$ are given by 
\begin{equation}\label{eq:asympB}
\Pb_a\left(\tau_0^-\wedge \tau_x^+>t\right) \equi_{t\rightarrow +\infty} \kappa \, e^{-\rho(x) t}
\end{equation}
where $\kappa>0$ and the function $\rho$ is defined by 
\[\rho(x) = \inf\{q\geq 0,\, W^{(-q)}(x)=0\}.\]
Note that $W^{(-q)}(x)$ denotes here the analytic extension of the scale function $q\rightarrow W^{(q)}(x)$ defined in (\ref{eq:W}).
It is also known,  from Lambert \cite[Prop. 5.1]{Lam}, that $\rho$ is continuous and strictly decreasing, hence it converges towards 
\[\lim_{x\rightarrow +\infty} \rho(x) = \inf_{x\geq 0} \rho(x)  = \rho(\infty).\]
Let $\varepsilon>0$ and take $x$ large enough such that
$\rho(x) \leq \rho(\infty)+\varepsilon$. 
By assumption, there exists ${A(x)}>0$ such that for $t\geq {A(x)}$, we have $v(x,t)<\varepsilon$. As a consequence,
\begin{equation}\label{eq:boundv-}
v(a,t) \geq e^{\beta(1-\varepsilon) t - \beta \int_0^{A(x)} v(x,s)ds  +{A(x)} \beta \varepsilon }\Pb_a\left(\tau_0^-\wedge \tau_x^+>t\right).
\end{equation}
Plugging (\ref{eq:asympB}) into (\ref{eq:boundv-}) and passing to the limit as $t\rightarrow +\infty$, we deduce that $\beta(1-\varepsilon) - \rho(x) <0$ which implies that 
$\beta < \rho(\infty)+ (1+\beta) \varepsilon. $}
Taking $\varepsilon$ small enough leads thus to the inequality $\beta \leq \rho(\infty)$ and the result will follow from the fact that $\rho(\infty)=q_\ast$ as proven in the next lemma.
\begin{lemma}\label{lem:KY}
Assume that $\Psi^\prime(0^+)<0$. Then:
\[\rho(\infty) = q_\ast = \sup\{q>0,  \,W^{(-q)}(x)>0\; \text{for every }x>0\}.\]
\end{lemma}

Note that this lemma  is very similar to Theorem 1.5 of  Yamato \cite{Yam}, in which the author studies the quasi-stationary distributions of $-L$ and assumes implicitly that $\Psi^\prime(0^+)>0$.

\begin{proof}
Observe first that comparing (\ref{eq:asympT0}) and (\ref{eq:asympB}) immediately yields $q_\ast \leq \rho(\infty)$. Also, by definition of $\rho(x)$ as an infimum, we have $W^{(-\rho(\infty))}(x)>0$ for every $x>0$. As a consequence,
\[\rho(\infty) \leq \sup\{q>0, \,W^{(-q)}(x)>0 \text{ for every }x>0\}.\]
To prove the converse inequality, recall from \cite[Chapter 8]{Kyp} that for $q\geq0$ and $a\leq x$, 
\[\E_a\left[e^{-q \tau_x^+}1_{\{\tau_x^+<+\infty\}}\right] = e^{-\Phi(q)(x-a)}.\]
In particular, letting $q\downarrow0$, we deduce that $\Pb_a(\tau_x^+<+\infty) = e^{-\Phi(0)(x-a)}$
hence
\[ \E_0\left[  e^{-\lambda \tau_x^+} \big| \tau_x^+<+\infty \right]  = \frac{ \E_0\left[  e^{-\lambda \tau_x^+}\right]}{\Pb_0(\tau_x^+<+\infty) } = e^{-\Phi(\lambda)x + \Phi(0)x}. \]
Recalling now that Formula (\ref{eq:W}) remains valid for $q\in \mathbb{C}$ and $\lambda >\Phi(|q|)$, we deduce by taking $\lambda$ large enough, since $\Phi$ is increasing, that 
\[
  \int_0^{+\infty}  e^{-\Phi(0) x} W^{(-q)}(x)  \E_0\left[  e^{-\lambda \tau_x^+}\big| \tau_x^+ <+\infty   \right] dx=     \int_0^{+\infty} e^{-\Phi(\lambda) x} W^{(-q)}(x)dx    =  \frac{1}{\lambda+q}.\]
Note that the random variable $\tau_x^+|\tau_x^+<+\infty$ admits a density from Kendall's identity since we have assumed that $L_t$ is absolutely continuous for every $t>0$. Assume now that  $q$ is such that $W^{(-q)}(x)>0$ for every $x>0$.  Applying the Fubini-Tonelli theorem and inverting this Laplace transform gives for a.e. $t>0$:
\[ \int_0^{+\infty} e^{-\Phi(0) x} W^{(-q)}(x) \Pb_0(\tau_x^+ \in dt | \tau_x^+<+\infty) dx = e^{-q t}dt\]
which implies that $q \leq q_\ast$ since conditionally on $\{\tau_x^+<+\infty\}$, the random variable $\tau_x^+$ admits exponential moments of at most $q_\ast$. 
\end{proof}

\subsection{Proof of the ``only if" part of Theorem \ref{theo:1} when $\Psi^\prime(0^+)\geq 0$}
We shall now assume that $\Psi^\prime(0^+)\geq 0$ and prove that there is a strictly positive probability that the branching process does not die out. Note that when $\Psi^\prime(0^+)>0$, as the process $L$ goes a.s. towards $+\infty$, one may simply bound the probability that $\zeta$ is infinite by the probability that one particle never goes below 0: 
\[\Pb_a( \zeta=\infty) \geq \Pb_a(\tau_0^- = +\infty) = \Psi^\prime(0^+)W^{(0)}(a)>0.\]
Assume now that $\Psi^\prime(0^+)=0$. The function $\Psi^\prime$ being continuous and strictly increasing, take $\lambda_\#>0$ small enough such that $0<\Psi^\prime(\lambda_\#)\lambda_\# - \Psi(\lambda_\#)<\beta$. Consider then the branching process $X^\#$, with the same branching mechanism as $X$, but in which the particles evolve as spectrally negative L\'evy processes $L^\#= (L_t - \Psi^\prime(\lambda_\#)t,\, t\geq0)$.
The Laplace exponent of $L_\#$ is given by 
\[\Psi_\#(\lambda) = \Psi(\lambda) - \Psi^\prime(\lambda_\#) \lambda\]
and satisfies $\Psi_\#^\prime(0^+)<0$ and $\Psi_\#^{\prime}(\lambda_\#) =0$. As a consequence, the minimum of $\Psi_\#$ equals $\Psi_\#(\lambda_\#) =  \Psi(\lambda_\#) - \Psi^\prime(\lambda_\#)\lambda_\# > -\beta$. From the first part of the proof, this implies that the branching process $X^\#$ does not die a.s., i.e., $\Pb_a(\zeta_\#=+\infty)>0$ {where $\zeta_\#$ denotes the extinction time of $X^\#$.} Finally, by coupling, as we have added a negative drift, it holds $\Pb_a(\zeta=+\infty)\geq \Pb_a(\zeta_\#=+\infty)>0$ which ends the proof of the converse part of Theorem \ref{theo:1}. \qed

\subsection{Proof of Formula (\ref{eq:asympzeta})}

Observe first that from (\ref{eq:boundv}) and (\ref{eq:asympT0}), we have
\[\frac{1}{t}\ln\Pb_a(\zeta>t) \leq  \beta+ \frac{1}{t}\ln\Pb_a(\tau_0^-\geq t)  \xrightarrow[t\rightarrow +\infty]{} \beta - q_\ast.\]
Conversely, from (\ref{eq:boundv-}),
\[\frac{1}{t}\ln\Pb_a(\zeta>t) \geq \beta(1-\varepsilon)  + \frac{\beta}{t}\left(A(x)  \varepsilon- \int_0^{A(x)} v(x,s)ds \right) + \frac{1}{t}\ln\Pb_a\left(\tau_0^-\wedge \tau_x^+>t\right).\]
Take $x$ large enough such that $\rho(x) \leq q_\ast +\varepsilon$. Passing to the limit as $t\rightarrow +\infty$ and using (\ref{eq:asympB}) yields
\[\liminf_{t\rightarrow +\infty} \frac{1}{t}\ln\Pb_a(\zeta>t) \geq  \beta(1-\varepsilon) -\rho(x) \geq \beta-q^\ast - \varepsilon(1+\beta),\]
and the result follows by letting $\varepsilon \downarrow 0$. \qed

\section{Study of the maximum}

\subsection{An integral equation }

We set 
\[u(a,x) = \Pb_a\left(\textbf{M}\geq x\right)\]
and first prove a Feynman-Kac-like representation for $u$ as in Proposition \ref{prop:zeta}. When dealing with extreme values of branching processes, this idea already appears in Lalley \& Shao \cite{LaSh} in their study of the maximum of a critical symmetric stable branching process.

\begin{proposition}\label{prop:M}
{When $\Psi^\prime(0^+)<0$ and  $\beta\leq q_\ast$,} the function $u$ admits the representation:
\[u(a,x) =  \E_a\left[1_{\{\tau_x^+ < \tau_0^-\}} e^{ \beta \int_0^{\tau_x^+} (1-u(L_{ r }, x))dr}\right],\qquad x\geq a. \]
\end{proposition}

\begin{proof}

The proof is similar to that of Proposition \ref{prop:zeta}. Applying the Markov property at the first branching event, we first write:
\begin{align*}
1-u(a,x) &= \Pb_a(\textbf{M}< x, \tau_0^- \leq \e) +  \Pb_a(\textbf{M}<x, \tau_0^- >\e) \\
& =  \Pb_a\left(\tau_0^-< \tau_x^+, \tau_0^-\leq\e\right)+    \E_a\left[1_{\{ \tau_0^-\wedge \tau_x^+>\e\}}\; (1-  u(L_\e, x))^2\right].
\end{align*}
Developing the square yields the non-linear integral equation 
\[u(a,x) = \Pb_a\left(\tau_x^+ \leq \e\wedge \tau_0^-\right) + 2 \E_a\left[1_{\{\tau_0^-\wedge \tau_x^+>\e\}} u(L_\e, x)\right] -  \E_a\left[1_{\{\tau_0^-\wedge \tau_x^+>\e\}}  u^2(L_\e, x)\right].\]
Define the linear operator 
\[T[f](a,x) = \E_a\left[1_{\{\tau_0^-\wedge \tau_x^+>\e\}}  (2-u(L_\e, x)) f(L_\e, x)\right]\]
so that $u$ is a solution of the Fredholm integral equation 
\[u(a,x) =\Pb_a\left(\tau_x^+ \leq \e\wedge \tau_0^-\right)  + T[u](a,x).\]
Setting to simplify the notations
\[ \varphi(a,x) =\Pb_a\left(\tau_x^+ \leq \e\wedge \tau_0^-\right) = \E_a\left[1_{\{\tau_x^+<\tau_0^-\}} e^{-\beta \tau_x^+}\right] \] 
we obtain by iteration that for $n\geq 2$, 
\begin{equation}\label{eq:conv}
u(a,x) =  \varphi(a,x) + \sum_{k=1}^{n-1} T^{\circ(k)}[\varphi](a,x) + T^{\circ(n)}[u](a,x),
\end{equation}
and it remains to evaluate the convolutions. Following the same proof as for Proposition \ref{prop:zeta}, we deduce that for any $k\geq 1$:
\begin{equation}\label{eq:TM}
T^{\circ(k)}[f](a,x) =  \frac{\beta^{k-1}}{(k-1)!} \E_a\left[1_{\{\tau_0^-\wedge \tau_x^+>\e\}}   f\left( L_{\e},x\right) (2-u(L_\e, x))   \left( \int_0^{\e} (2-u(L_{r }, x))dr\right)^{k-1}\right].
\end{equation}
In particular, taking $f=u$, we obtain since $u$ is positive and bounded by 1,  and $\e$ is independent of $\tau_0^-\wedge \tau_x^+$,
\begin{align*}
T^{\circ(n)}[u](a,x) &\leq  2 \frac{(2\beta)^{n-1}}{(n-1)!}    \E_a\left[\e^{n-1} 1_{\{\tau_0^-\wedge \tau_x^+>\e\}} \right]  \\
&=  \frac{(2\beta)^{n}}{(n-1)!} 
\int_0^{+\infty}\Pb_a( \tau_0^-\wedge \tau_x^+> z)  z^{n-1}e^{- \beta z} dz.
\end{align*}
From the asymptotics (\ref{eq:asympB}), there exist two constants $A, C>0$ such that:
\begin{align*}
T^{\circ(n)}[u](a,x)& \leq  \frac{(2\beta)^{n-1}}{(n-1)!} \left(\int_0^A  z^{n-1} e^{-\beta z} dz+ \int_A^{+\infty} z^{n-1} e^{-\beta z}  \,e^{-\rho(x) z} dz\right)\\
 & \leq \frac{(2\beta)^{n-1}}{(n-1)!}      \left( \frac{A^n}{n}+  C \frac{(n-1)!}{(\beta + \rho(x))^{n}}\right)\xrightarrow[n\rightarrow+\infty]{}0
\end{align*}
since, {from Lemma \ref{lem:KY},} $\beta \leq q_\ast < \rho(x)$ for any fixed $x>0$.
As a consequence, we obtain the series formula 
\[u(a,x) =  \varphi(a,x) + \sum_{k=1}^{+\infty} T^{\circ(k)}[\varphi](a,x).\]
Plugging (\ref{eq:TM}) into this series representation and computing the sum  yields, after another application of the Markov property,  
\begin{align*}
u(a,x) &= \Pb_a\left(\tau_x^+ \leq \e\wedge \tau_0^-\right)  + \E_a\left[   1_{\{ \tau_x^+ <\tau_0^-\}}  1_{\{\tau_x^+\geq \e\}} e^{-\beta (\tau_x^+-\e)}(2-u(L_\e, x))  e^{\beta \int_0^{\e}(2-u(L_r, x)) dr  }\right]\\
&=\Pb_a\left(\tau_x^+ \leq \e\wedge \tau_0^-\right)+\int_0^{+\infty} \E_a\left[  1_{\{\tau_x^+< \tau_0^-\}} 1_{\{ \tau_x^+ \geq s\}}   e^{-\beta \tau_x^+} (2-u(L_s, x)) e^{\beta \int_0^{s}(2-u(L_r, x))dr  }\right] \beta  ds\\
&=\E_a\left[  1_{\{\tau_x^+< \tau_0^-\}}   e^{\beta \int_0^{\tau_x^+}(1-u(L_r, x))dr  }\right],
\end{align*}
which ends the proof of Proposition \ref{prop:M}.
\end{proof}

\begin{remark}
If we denote by $T_x$ the first time when a particle of the branching process $X$ hits the level $x$, we have $\Pb_a(\textbf{M}\geq x) = \Pb_a(\zeta > T_x)$. As a consequence, it is not surprising that the representation result of Proposition \ref{prop:M} is similar to that of Proposition \ref{prop:zeta} but with $t$ replaced by $\tau_x^+$.
\end{remark}

\subsection{Proof of Theorem \ref{theo:2}}
We now study the limit of $u(a,x)$ as $x\rightarrow +\infty$, and thus assume throughout this section that $x\geq a$. Let us recall the classical Esscher transform. 
For $c\geq -q_\ast$, we denote by $\Pb_a^{(c)}$ the probability defined by
\begin{equation}\label{eq:change}
\frac{\text{d}\Pb_a^{(c)}}{\text{d}\Pb_a}\Bigg|_{\F_t} = e^{\Phi(c) (L_t-a) -c t},
\end{equation}
{where $(\F_t)$ denotes the natural filtration of $L$.} Under $\Pb_a^{(c)}$, the process $L$ is still a spectrally negative L\'evy process starting from $a$ but with Laplace exponent $\Psi_{c}(\lambda) = \Psi(\lambda +\Phi(c)) -c$. In particular,  since $\Psi_{c}^\prime(0) = \Psi^\prime(\Phi(c))\geq 0$, the process $L$ no longer drifts a.s. towards $-\infty$, {see Bertoin \cite[Chapter VII, Corollary 2]{Ber}.
Denoting by $W_{c}^{(q)}$ its scale function, {we have from (\ref{eq:W}) for $q\geq0$:}
\[\int_0^{+\infty}e^{-\lambda x} W_{c}^{(q)}(x)dx = \frac{1}{ \Psi(\lambda +\Phi(c)) -c-q},\qquad \lambda > \Phi(q+c) - \Phi(c),\]
which implies, comparing the definition of $W_{c}^{(q)}$ and $W^{(q+c)}$, that
\[W_{c}^{(q)}(x) = e^{\Phi(c) x}W^{(c+q)}(x), \qquad x\geq0.\] }
Applying (\ref{eq:change}) to Proposition \ref{prop:M} and using the absence of positive jumps yields
\[e^{\Phi(-\beta)a}  \E_a^{(-\beta)}\left[ e^{-\int_0^{\tau_x^+} u(L_r,x) dr} 1_{\{\tau_x^+<\tau_0^-\}}\right] =  \E_a\left[e^{\Phi(-\beta)x  + \beta \tau_x^+} e^{-\int_0^{\tau_x^+} u(L_r,x) dr}1_{\{\tau_x^+<\tau_0^-\}}    \right]\]
 {where $ \E_a^{(-\beta)}$ denotes the expectation under $ \Pb_a^{(-\beta)}$.}
Using the invariance by translation of L\'evy processes, this is further equal to
\begin{equation}\label{eq:appchange}
e^{\Phi(-\beta)(a-x)}u(a,x) =  
\E_0^{(-\beta)}\left[e^{-\beta \int_0^{\tau_{x-a}^+} u(a+L_r,x)dr}1_{\{\tau_{x-a}^+ < \tau_{-a}^-\}}  \right]. 
\end{equation}
Note that under $\Pb_0^{(-\beta)}$ the random variable $\tau_{x-a}^+$ is a.s. finite. As a consequence, one may apply the time reversal result of  Bertoin \cite[Chapter VII, Th. 18]{Ber} to obtain 
 \begin{equation}\label{eq:tg}
 \E_0^{(-\beta)}\left[e^{-\beta \int_0^{\tau_{x-a}^+} u(a+L_r,x)dr}1_{\{\tau_{x-a}^+ < \tau_{-a}^-\}}  \right]=\E_0^{(-\beta)\uparrow}\left[e^{-\beta \int_0^{g_{x-a}^-} u(x-L_{r},x)dr}1_{\{g_{x-a}^- < \tau_{x}^+\}}  \right]
 \end{equation}
 where $g_x^-$ denotes the last passage time of $L$ below the level $x$, i.e., $g_x^- = \sup\{t\geq0,\, L_t\leq x\}$, and $\Pb_a^{(-\beta)\uparrow}$ denotes the law of $L$ (under $\Pb_a^{(-\beta)}$) conditioned to stay positive, which is defined by:
 \begin{equation}\label{eq:Wup}
 \forall \Lambda_t \in \F_t,\qquad \quad \Pb_a^{(-\beta)\uparrow}(\Lambda_t) = \frac{1}{W_{-\beta}^{(0)}(a)}\E_a^{(-\beta)}\left[W_{-\beta}^{(0)}(L_t) 1_{\{\tau_0^->t\}} 1_{\Lambda_t}\right].\end{equation}
Applying the strong Markov property, we then obtain the lower bound
\begin{multline}\label{eq:lower}
 \E_0^{(-\beta)\uparrow}\left[e^{-\beta \int_0^{g_{x-a}^-} u(x-L_{r},x)dr}1_{\left\{\inf\limits_{s\geq \tau_x^+}L_s >x-a   \right\}}  \right]\\\geq \E_0^{(-\beta)\uparrow}\left[e^{-\beta \int_0^{\tau_x^+} u(x-L_{r},x)dr}\right]\Pb_x^{(-\beta)\uparrow}\left(\inf\limits_{s\geq 0}L_s >x-a   \right).
\end{multline}
To compute the last term, we come back to the absolute continuity formula (\ref{eq:Wup}):
\[ \Pb_x^{(-\beta)\uparrow}\left(\inf\limits_{s\leq t}L_s >x-a   \right) = \Pb_x^{(-\beta)\uparrow}\left(\tau_{x-a}^-> t \right)=\E_x^{(-\beta)}\left[ \frac{W_{-\beta}^{(0)}(L_{t})}{W_{-\beta}^{(0)}(x)}1_{\{t<\tau_0^-\}} 1_{\{t< \tau_{x-a}^-\}}    \right].\]
Since  $x> x-a\geq 0$, we have $\{t<\tau_0^-\}\cap\{t<\tau_{x-a}^-\} = \{t<\tau_{x-a}^-\}$, and by translation, this is further equal to
\[\E_a^{(-\beta)}\left[ \frac{W_{-\beta}^{(0)}(x-a+L_{t})}{W_{-\beta}^{(0)}(x)}1_{\{t<\tau_{0}^-\}}    \right]  = \frac{W_{-\beta}^{(0)}(a)}{W_{-\beta}^{(0)}(x)}\E_a^{(-\beta)\uparrow}\left[ \frac{W_{-\beta}^{(0)}(x-a+L_{t})}{W_{-\beta}^{(0)}(L_t)}   \right].\]
We now let $t\rightarrow+\infty$ and use the fact that $L$ goes to $+\infty$ a.s. under $\Pb_a^{(-\beta)\uparrow}$ {(see Bertoin \cite[Chapter VII, Lemma 12]{Ber})}. Recall from Hubalek \& Kyprianou \cite[Section 3]{HuKy} that the asymptotics of $W_{-\beta}^{(0)}$ are given by 
\begin{equation}\label{eq:asympW}
\begin{cases}
\displaystyle \lim_{z\rightarrow +\infty} W_{-\beta}^{(0)}(z) =\frac{1}{\Psi^\prime_{-\beta}(0^+)} =\frac{1}{ \Psi^\prime(\Phi(-\beta))}<+\infty  &\qquad \text{ if }\beta<q_\ast,\\
\displaystyle W_{-q_\ast}^{(0)}(z) \equi_{z\rightarrow+\infty} \frac{2}{\Psi^{\prime\prime}(\lambda_\ast)} z &\qquad \text{ if }\beta=q_\ast.
\end{cases}
\end{equation}
As a consequence, we obtain in both cases, {applying the dominated convergence theorem,}
\begin{equation}\label{eq:Pbuparrow}
\Pb_x^{(-\beta)\uparrow}\left(\inf\limits_{s\geq 0}L_s >x-a   \right) = \frac{W_{-\beta}^{(0)}(a)}{W_{-\beta}^{(0)}(x)}.  
\end{equation}
{Plugging together (\ref{eq:appchange}), (\ref{eq:tg}), (\ref{eq:lower}) and  (\ref{eq:Pbuparrow})}, we thus deduce from Fatou's lemma that
\begin{equation*}
\liminf_{x\rightarrow+\infty} W_{-\beta}^{(0)}(x)e^{\Phi(-\beta)(a-x)}u(a,x) \geq W_{-\beta}^{(0)}(a) \E_0^{(-\beta)\uparrow}\left[   \liminf_{x\rightarrow +\infty}e^{-\beta \int_0^{+\infty} u(x-L_r,x)dr}\right].
\end{equation*}
Observe next that by translation, $u(x-a,x)$ corresponds to the probability that a branching process starting at $0$, and where the particles are killed at the level $a-x$, reaches the level $a$ before dying. Letting $x\rightarrow +\infty$, we deduce that 
 $u(x-a,x) \xrightarrow[x\rightarrow +\infty]{}\Pb_0( \mathcal{M}\geq a)$
 where $\mathcal{M}$ is the all-time maximum of a free spectrally negative branching L\'evy process. By monotone convergence, we have thus obtained 
 \begin{equation}\label{eq:Binf}
\liminf_{x\rightarrow+\infty} W_{-\beta}^{(0)}(x)e^{\Phi(-\beta)(a-x)}u(a,x) \geq W_{-\beta}^{(0)}(a) \E_0^{(-\beta)\uparrow}\left[   e^{-\beta \int_0^{+\infty} \Pb_0(\mathcal{M}\geq L_r) dr}\right]
\end{equation}
 which gives a lower bound, provided the expectation on the right-hand side is non-null. \\
 
  To get an upper bound, we start back from (\ref{eq:tg}) and take $A>0$. On the one hand, we write:
  \begin{align*}
  \E_0^{(-\beta)\uparrow}\left[e^{-\beta \int_0^{g_{x-a}^-} u(x-L_{r},x)dr}1_{\left\{\inf\limits_{s\geq \tau_x}L_s >x-a   \right\}} 1_{\{g_{x-a}^-\leq A\}} \right]&\leq 
  \Pb_0^{(-\beta)\uparrow}\left(g_{x-a}^-\leq A\right)
\end{align*}
On the other hand, we have 
 \begin{align*}
  \E_0^{(-\beta)\uparrow}\left[e^{-\beta \int_0^{g_{x-a}^-} u(x-L_{r},x)dr}1_{\left\{\inf\limits_{s\geq \tau_x^+}L_s >x-a   \right\}} 1_{\{g_{x-a}^-> A\}} \right]&\leq 
  \E_0^{(-\beta)\uparrow}\left[e^{-\beta \int_0^{A} u(x-L_{r},x)dr}1_{\left\{\inf\limits_{s\geq \tau_x^+}L_s >x-a   \right\}} \right].
\end{align*}
Decomposing further according to whether $\tau_x^+\leq A$ or $\tau_x^+>A$ and applying the strong Markov property, we deduce that 
\begin{multline*}\E_0^{(-\beta)\uparrow}\left[e^{-\beta \int_0^{A} u(x-L_{r},x)dr}1_{\left\{\inf\limits_{s\geq \tau_x^+}L_s >x-a   \right\}}  \right]\\
\leq  \Pb_0^{(-\beta)\uparrow}(\tau_x^+\leq A) +   \E_0^{(-\beta)\uparrow}\left[e^{-\beta \int_0^{A}u(x-L_r, x)dr} 1_{\{A< \tau_x^+\}}\right]\Pb_x^{(-\beta)\uparrow}\left(\inf\limits_{s\geq 0}L_s >x-a   \right).
\end{multline*}
Multiplying both sides by $W_{-\beta}^{(0)}(x)$ and letting $x\rightarrow +\infty$, we obtain using (\ref{eq:Pbuparrow}) and the dominated convergence theorem that :
\begin{multline}\label{eq:Bsup}
\limsup_{x\rightarrow+\infty} W_{-\beta}^{(0)}(x)e^{\Phi(-\beta)(a-x)}u(a,x)\\ 
\leq \limsup_{x\rightarrow+\infty} W_{-\beta}^{(0)}(x)   \left( \Pb_0^{(-\beta)\uparrow}\left(g_{x-a}^-\leq A\right) +  \Pb_0^{(-\beta)\uparrow}(\tau_x^+\leq A)\right) \\+ 
\E_0^{(-\beta)\uparrow}\left[e^{-\beta \int_0^{A} \Pb_0\left(\mathcal{M}\geq L_r\right)dr}\right]W_{-\beta}^{(0)}(a) 
\end{multline}
and it remains to show that the first term on the right-hand side is null. When $\beta<q_\ast$, this is a consequence of (\ref{eq:asympW}) and of the fact that  $L$ goes a.s. to $+\infty$ under $\Pb_0^{(-\beta)\uparrow}$.
When $\beta=q_\ast$ a little more work is required as $W_{-q_\ast}^{(0)}$ goes linearly to $+\infty$. On the one hand, the time reversal argument of \cite[Theorem 18, p.206]{Ber} implies that
\[\Pb_0^{(-q_\ast)\uparrow}(g_{x-a}^-\leq A) = \Pb_0^{(-q_\ast)}(\tau_{x-a}^+\leq A) \leq e^{q_\ast A} \E_0^{(-q_\ast)}\left[e^{-q_\ast\tau_{x-a}^+}\right]= e^{q_\ast A} e^{-(x-a) (\Phi(0)-\Phi(-q_\ast))}  \]
so it indeed holds that
$\displaystyle \lim\limits_{x\rightarrow+\infty} W_{-q_\ast}^{(0)}(x) \Pb_0^{(-q_\ast)\uparrow}(g_{x-a}^-\leq A) =0.$
On the other hand, we may write similarly
\[\limsup\limits_{x\rightarrow+\infty} W_{-q_\ast}^{(0)}(x) \Pb_0^{(-q_\ast)\uparrow}(\tau_x^+\leq A) \leq e^{q_\ast A}\limsup\limits_{x\rightarrow+\infty} W_{-q_\ast}^{(0)}(x)  \E_0^{(-q_\ast)\uparrow}\left[e^{-q_\ast \tau_x^+}\right] 
\]
which leads us to study the Laplace transform of $\tau_x^+$ under $\Pb_0^{(-q_\ast)\uparrow}$.
To do so, take $\varepsilon>0$ and observe that using the absolute continuity formula (\ref{eq:Wup})
\[\E_\varepsilon^{(-q_\ast)\uparrow}\left[e^{-q_\ast \tau_x^+\wedge t}\right] = \frac{1}{W_{-q_\ast}^{(0)}(\varepsilon)}\E_\varepsilon^{(-q_\ast)}\left[W_{-q_\ast}^{(0)}(L_{\tau_x^+\wedge t})e^{-q_\ast \tau_x^+\wedge t}  1_{\{\tau_0^- >\tau_x^+\wedge t\}}\right].\]
Letting $t\rightarrow+\infty$ and applying the dominated convergence theorem since $W_{-q_\ast}^{(0)}$ is increasing and $L$ has no positive jumps, we deduce from \cite[Theorem 8.1]{Kyp} that
\begin{equation}\label{eq:Evareps}
\E_\varepsilon^{(-q_\ast)\uparrow}\left[e^{-q_\ast \tau_x^+}\right]  = \frac{W_{-q_\ast}^{(0)}(x)}{W_{-q_\ast}^{(0)}(\varepsilon) }  \E_\varepsilon^{(-q_\ast)}\left[e^{-q_\ast \tau_x^+} 1_{\{\tau_0^->\tau_x^+\}}\right]  
=\frac{W_{-q_\ast}^{(0)}(x)}{W_{-q_\ast}^{(0)}(\varepsilon) } \frac{W^{(q_\ast)}_{-q_\ast}(\varepsilon)}{W^{(q_\ast)}_{-q_\ast}(x)}.
\end{equation}
Recall next the bound for $\varepsilon>0$, see the proof of \cite[Lemma 8.3]{Kyp}, 
\[\frac{W^{(q_\ast)}_{-q_\ast}(\varepsilon)}{W_{-q_\ast}^{(0)}(\varepsilon)} \leq  \sum_{k\geq0} q_\ast^{k} \frac{\varepsilon^k}{k!}\left(W_{-q_\ast}^{(0)}(\varepsilon)\right)^{k}.\]
Then, letting $\varepsilon\downarrow0$ in (\ref{eq:Evareps}), we conclude from the asymptotics in \cite[Section 3]{HuKy} that 
\begin{align*}
\limsup\limits_{x\rightarrow+\infty} W_{-q_\ast}^{(0)}(x) \Pb_0^{(-q_\ast)\uparrow}(\tau_x^+\leq A)  &\; \leq\; e^{q_\ast A}   \limsup\limits_{x\rightarrow+\infty} W_{-q_\ast}^{(0)}(x) \frac{W_{-q_\ast}^{(0)}(x)}{W^{(q_\ast)}_{-q_\ast}(x) }\\
&\equi_{x\rightarrow+\infty} e^{q_\ast A}   \Psi^\prime(\Phi(0)) (W_{-q_\ast}^{(0)}(x))^2 e^{-x (\Phi(0)-\Phi(-q_\ast))} 
 \end{align*}
which also converges to 0. As a consequence, we have obtained the bound for $\beta\leq q_\ast$:
\begin{equation}\label{eq:Bsup2}
\limsup_{x\rightarrow+\infty} W_{-\beta}^{(0)}(x)e^{\Phi(-\beta)(a-x)}u(a,x)\leq
\E_0^{(-\beta)\uparrow}\left[e^{-\beta \int_0^{A} \Pb_0\left(\mathcal{M}\geq L_r\right)dr}\right]W_{-\beta}^{(0)}(a). 
\end{equation}
Letting $A\uparrow +\infty$ in (\ref{eq:Bsup2}) and gathering (\ref{eq:Binf}) and (\ref{eq:Bsup2}), we have thus proven that
\[\lim_{x\rightarrow+\infty} W_{-\beta}^{(0)}(x)e^{\Phi(-\beta)(a-x)}u(a,x) = W_{-\beta}^{(0)}(a) \E_0^{(-\beta)\uparrow}\left[e^{-\beta \int_0^{+\infty} \Pb_0\left(\mathcal{M}\geq L_r\right)dr}\right]\]
which is Theorem \ref{theo:2} where, from (\ref{eq:asympW}), 
\begin{multline*}
 \kappa_\beta =  \Psi^\prime(\Phi(-\beta)) \E_0^{(-\beta)\uparrow}\left[e^{-\beta \int_0^{+\infty} \Pb_0\left(\mathcal{M}\geq L_r\right)dr}\right]  \\\text{and}\qquad \kappa_{q_\ast} = \frac{\Psi^{\prime\prime}(\lambda_\ast)}{2} \E_0^{(-q^\ast)\uparrow}\left[e^{-q^\ast \int_0^{+\infty} \Pb_0\left(\mathcal{M}\geq L_r\right)dr}\right].\end{multline*}
\noindent
It finally remains to check that the expectations on the right-hand sides are not null. To do so, we first prove that $\mathcal{M}$ is a.s. finite.
Since $u$ is positive, we deduce from Proposition \ref{prop:M} that
 \begin{equation}\label{eq:uxax}
 u(x-a,x) \leq \E_{x-a}\left[1_{\{\tau_x^+<\tau_0^-\}}e^{\beta \tau_x^+}\right]. 
 \end{equation}
Then, from the Esscher transform (\ref{eq:change}), we have since $\beta \leq q_\ast$:
\[  \E_{x-a}\left[e^{\Phi(-\beta) (L_{\tau_x^+}-(x-a)) +\beta \tau_x^+} 1_{\{\tau_x^+<\tau_0^-\}}1_{\{\tau_x^+\leq t\}}\right]\leq \Pb_{x-a}^{(-\beta)}\left(\tau_x^+<\tau_0^-\wedge t\right).\]
Passing to the limit as $t\rightarrow+\infty$, we deduce from the monotone convergence theorem  that 
\[ e^{\Phi(-\beta) a}   \E_{x-a}\left[e^{\beta \tau_x^+} 1_{\{\tau_x^+<\tau_0^-\}}\right]\leq \Pb_{x-a}^{(-\beta)}\left(\tau_x^+<\tau_0^-\right) \leq 1 .\]
Then, going back to (\ref{eq:uxax}) and letting $x\rightarrow +\infty$, we conclude that
 \begin{equation}\label{eq:boundMf}
 \Pb_0( \mathcal{M}\geq a) \leq e^{- \Phi(-\beta)a}<1
 \end{equation}
 which shows that $\mathcal{M}$ is a.s. finite when $\beta\leq q_\ast$. Now, to prove that $\kappa_\beta$ is non-null, it is sufficient from (\ref{eq:boundMf}) to check that 
$\int_0^{+\infty} e^{-\Phi(-\beta)L_r}dr<+\infty$ a.s. under $\Pb_0^{(-\beta)\uparrow}$. Applying the Fubini-Tonelli theorem and Bertoin \cite[Chapter VII.3, Cor.16]{Ber}, we have 
\[\E_0^{(-\beta)\uparrow}\left[\int_0^{+\infty} e^{-\Phi(-\beta)L_r}dr\right] = \int_0^{+\infty} \int_0^{+\infty} e^{-\Phi(-\beta) x} \frac{x W_{-\beta}^{(0)}(x)}{r}\Pb_0^{(-\beta)}(L_r\in dx) dr.\] 
Using  Kendall's identity, $r\Pb_0^{(-\beta)}(\tau_x^+\in dr) dx=  x \Pb_0^{(-\beta)}(L_r\in dx)dr$, this is further equal to 
\[ \int_0^{+\infty} \int_0^{+\infty} e^{-\Phi(-\beta) x}  W_{-\beta}^{(0)}(x)\Pb_0^{(-\beta)}(\tau_x^+\in dt)  dx =  \int_0^{+\infty} e^{-\Phi(-\beta) x}  W_{-\beta}^{(0)}(x)dx<+\infty\]
since $\tau_x^+$ is a.s. finite under $\Pb_0^{(-\beta)}$ and  the asymptotics of $W_{-\beta}^{(0)}$ are given by (\ref{eq:asympW}).\qed

\subsection{Proof of Corollary \ref{cor}}

Let us set $m(a) = \Pb_0(\mathcal{M}>a)$. Replacing $a$ by $x-a$  in Proposition \ref{prop:M}, we deduce as before by translation that 
\[ u(x-a,x) =  \E_{x-a}\left[1_{\{\tau_x^+ < \tau_0^-\}} e^{ \beta \int_0^{\tau_x^+} (1-u(L_{ r }, x))dr}\right]= \E_{0}\left[1_{\{\tau_a^+ < \tau_{a-x}^-\}} e^{ \beta \int_0^{\tau_a^+} (1-u(x-a+L_{ r }, x))dr}\right].  \]
Passing to the limit as $x\rightarrow+\infty$ and applying the dominated convergence theorem  we conclude that
\[m(a) = \E_0\left[1_{\{\tau_a^+ <+\infty\}} e^{\beta \int_0^{\tau_a^+} (1-m(a-L_s))ds}   \right] = e^{-\Phi(-\beta) a} \E_0^{(-\beta)}\left[1_{\{\tau_a^+ <+\infty\}} e^{-\beta \int_0^{\tau_a^+} m(a-L_s)ds}    \right]\]
where the last equality follows from the Esscher transform.
Since $\Pb_0^{(-\beta)}(\tau_a^+<+\infty)=1$, the time reversal result \cite[Chapter VII, Th. 18]{Ber} yields 
\[m(a) =  e^{-\Phi(-\beta) a} \E_0^{(-\beta)\uparrow}\left[e^{-\beta \int_0^{g_a^-} m(L_s)ds}    \right].\]
The result then follows by letting $a\rightarrow +\infty$, since, as before, the expectation on the right-hand side converges towards a strictly positive constant.

\qed

\noindent
\textbf{Acknowledgments.}
We are grateful to the referee whose very careful reading and suggestions help to improve the presentation of the paper.

\end{document}